\documentclass[11pt,reqno]{amsart}

\usepackage{amscd,amssymb,amsmath,amsthm}
\usepackage[arrow,matrix]{xy}
\usepackage{graphicx}
\usepackage{epstopdf}
\usepackage{color}
\newtheorem{thm}{Theorem}
\newtheorem{defn}{Definition}
\newtheorem{lemma}{Lemma}

\newtheorem{rk}{Remark}
\newtheorem{cor}{Corollary}

\numberwithin{equation}{section} 

\begin{document}
\title[A regular gonosomal evolution operator]
{A Regular Gonosomal Evolution Operator with uncountable set of fixed points}

\author{A.T. Absalamov ~}
\address{ A.T. Absalamov, Samarkand State University,
Boulevard str., 140104,
Samarkand, Uzbekistan.}
\email{absalamov@gmail.com}
\author{U.A. Rozikov ~}
\address{U.A. Rozikov$^{a,b,c}$ \begin{itemize}
 \item[$^a$] V.I.Romanovskiy Institute of Mathematics of Uzbek Academy of Sciences;
\item[$^b$] AKFA University, 1st Deadlock 10, Kukcha Darvoza, 100095, Tashkent, Uzbekistan;
\item[$^c$] Faculty of Mathematics, National University of Uzbekistan.
\end{itemize}}
\email{rozikovu@yandex.ru}

\begin{abstract} In this paper we study dynamical systems generated by
a gonosomal evolution operator of a bisexual population.
We find explicitly all (uncountable set) of fixed points of the operator.
It is shown that each fixed point has eigenvalues less or equal to 1.
Moreover, we show that each trajectory converges to a fixed point, i.e. the operator is reqular. There are uncountable family of
invariant sets each of which consisting unique fixed point. Thus there is one-to-one correspondence between such invariant sets and the set of fixed points. Any trajectory  started at a point of the invariant
set converges to the corresponding fixed point.
\end{abstract}

\keywords{Dynamical systems; fixed point; invariant set, limit point.}
\subjclass[2010]{92D10 (34C05 37N25 92D15)}
\maketitle
\section{Introduction}

Population dynamics theory is important to a proper understanding of living populations
at all levels. This is a well developed branch of mathematical biology, which has a
history of more than two hundred years.

The book \cite{Ba} contains a short history of applications of mathematics to
solving various problems in population dynamics. For background and motivations of the theory of population dynamics see \cite{AR}-\cite{Varro.book.2015}.

In this paper we consider a bisexual population  which consists
females partitioned into types
indexed by $\{1,2,\dots,n\}$ and the males partitioned into types
indexed by $\{1,2,\dots,\nu \}$  (see \cite{Lyubich.book.1992}, \cite{Rpd}, \cite{Rozikov.book.2011} for details).

 Let $\gamma_{ik,j}^{(f)}$ and $\gamma_{ik,l}^{(m)}$ be inheritance coefficients
 defined as the probability that a female offspring is type $j$ and, respectively,
 that a male offspring is of type $l$, when the parental pair is
 $ik$ $(i,j=1,\dots,n$; and $k,l=1,\dots,\nu)$. These quantities satisfy the following
 \begin{equation}\label{kp}\begin{array}{ll}
\gamma_{ik,j}^{(f)}\geq 0, \ \ \gamma_{ik,l}^{(m)}\geq 0,\\[3mm]
 \sum_{j=1}^n\gamma_{ik,j}^{(f)}+\sum_{l=1}^\nu \gamma_{ik,l}^{(m)}=1, \ \ \mbox{for all} \ \ i,k,j,l.
\end{array}
\end{equation}
Define $(n+\nu-1)-$dimensional  simplex:
$$S^{n+\nu-1}=\left\{s=(x_1, \dots, x_n, y_1,\dots,y_\nu)\in \mathbb R^{n+\nu}: x_i\geq 0, y_j\geq 0, \sum_{i=1}^nx_i+\sum_{j=1}^\nu y_j=1\right\}.$$
Denote
$${\mathcal O}=\left\{s\in S^{n+\nu-1}: (x_1,\dots,x_n)=(0,\dots,0) \, \mbox{or} \, (y_1,\dots,y_\nu)=(0,\dots,0)\right\}.$$
$${\mathcal S}^{n,\nu}=S^{n+\nu-1}\setminus {\mathcal O}.$$

Following \cite{Rozikov.book.2016} define an evolution operator $V: {\mathcal S}^{n,\nu}\to {\mathcal S}^{n,\nu}$ (which is called
 normalized gonosomal operator) as
\begin{equation}\label{v3n}
V:\left\{ \begin{array}{ll}
x'_{j}\;=\;\dfrac{\sum_{i,k=1}^{n,\nu}\gamma_{ik,j}^{(f)}x_{i}y_{k}}{\Bigl(\sum_{i=1}^{n}x_{i}\Bigr)\left(\sum_{j=1}^{\nu}y_{j}\right)}, & j=1,\dots,n\medskip\\[3mm]
y'_{l}\;=\;\dfrac{\sum_{i,k=1}^{n,\nu}\gamma_{ik,l}^{(m)}x_{i}y_{k}}{\Bigl(\sum_{i=1}^{n}x_{i}\Bigr)\left(\sum_{j=1}^{\nu}y_{j}\right)}, & l=1,\dots,\nu.
\end{array}\right.
\end{equation}

{\bf The main problem:} For given operator $V$ and initial point $z^{(0)}\in {\mathcal S}^{n,\nu}$
what ultimately happens with the trajectory
$z^{(m)}=V(z^{(m-1)}), \ \ m=1,2,\dots$?
Does the limit $\lim_{m\to\infty} z^{(m)}$ exist?
If not what is the set of limit points of the sequence?

In general, this is very difficult problem.
In book \cite{Rpd} several recently obtained results related to this main problem are given.

In this paper we consider the special case: $n=\nu=2$ and
the following coefficients:
\begin{equation}\label{kk}
\begin{array}{cccc}
\gamma_{11,1}^{(f)}=a &\gamma_{11,2}^{(f)}=0& \gamma_{11,1}^{(m)}=b&\gamma_{11,2}^{(m)}=0\\[3mm]
\gamma_{12,1}^{(f)}=0 &\gamma_{12,2}^{(f)}=\sigma_1& \gamma_{12,1}^{(m)}=\sigma_2&\gamma_{12,2}^{(m)}=0\\[3mm]
\gamma_{21,1}^{(f)}=0 &\gamma_{21,2}^{(f)}=a& \gamma_{21,1}^{(m)}=b&\gamma_{21,2}^{(m)}=0\\[3mm]
\gamma_{22,1}^{(f)}=0 &\gamma_{22,2}^{(f)}=a& \gamma_{22,1}^{(m)}=0&\gamma_{22,2}^{(m)}=b.
\end{array}
\end{equation}
Then corresponding  evolution operator $W:S^{2,2}\rightarrow{S^{2,2}}$ is
\begin{equation}\label{A1}
W: \left\{
\begin{array}{ll}
\begin{aligned}
& x' &= \quad &\frac{axu}{(x+y)(u+v)} \\[2mm]
& y' &= \quad &\frac{\sigma_1xv+ayu+ayv}{(x+y)(u+v)} \\[2mm]
& u' &= \quad &\frac{\sigma_2xv+bxu+byu}{(x+y)(u+v)}\\[2mm]
& v' &= \quad &\frac{byv}{(x+y)(u+v)},
\end{aligned}
\end{array}
 \right.
\end{equation}
where coefficients satisfy
$$a+b=\sigma_1+\sigma_2=1, \ \ a, b, \sigma_1, \sigma_2>0.$$
\begin{rk} From the probabilities (\ref{kk}) one can notice that type $1$ of females (resp. type $2$ of males)
can be born only if both parents have type $1$ (resp. 2). Type $2$ of females (resp. type $1$ of males) can \emph{not} be
born if both parents have type $1$ (resp. $2$).
\end{rk}
For this operator $W$ and arbitrarily initial point
$s^{(0)}\in S^{2,2}$, we will study the trajectory
$\{s^{(m)}\}_{m=0}^{\infty}$, where $$s^{(m)}=W^{m}(s^{(0)})=
\underbrace {W(W(...W}_{m}(s^{(0)}))...).$$
\section{Fixed points}
\noindent A point $s$ is called a fixed point of the operator $W$
if $s=W(s)$. The set of all fixed points denoted by Fix$(W)$.

Let us find all the fixed points of $W$
given by \eqref{A1}, i.e.~we solve the following system of
equations for $(x,y,u,v)$:
\begin{equation}\label{F}
\left\{
\begin{array}{ll}
x(x+y)(u+v)=axu,\vspace{1,5mm} \cr
y(x+y)(u+v)=\sigma_1xv+ayu+ayv,\vspace{1,5mm} \cr
u(x+y)(u+v)=\sigma_2xv+bxu+byu,\vspace{1,5mm} \cr
v(x+y)(u+v)=byv.
\end{array}
\right.
\end{equation}
If $x=0$ then $y\neq0$ and from the second equation of the
system \eqref{F} we get $y=a$. In addition, the third and the fourth
equations of the system \eqref{F} give $u+v=b$.

If $v=0$ then $u\neq0$ and from the third equation of the
system \eqref{F} we get $u=b$. The second and the third
equations of the system \eqref{F} give $x+y=a$.

If $xv\neq0$ then we come to
\begin{equation}\label{I.par}
\left\{
\begin{array}{ll}
(x+y)(u+v)=au,\vspace{1,5mm} \cr
y(x+y)(u+v)=\sigma_1xv+ayu+ayv,\vspace{1,5mm} \cr
u(x+y)(u+v)=\sigma_2xv+bxu+byu,\vspace{1,5mm} \cr
(x+y)(u+v)=by.
\end{array}
\right.
\end{equation}

The first and the second equations of the system \eqref{I.par}
give $$\sigma_1x+ay=0.$$ At the same time the third and the fourth
equations of the system \eqref{I.par} give $$\sigma_2v+bu=0.$$
Since $a,b,\sigma_1,\sigma_2>0$ then when we solve the last two
equations we obtain $x=y=u=v=0$, however this point is not in
the space $S^{2,2}$. Thus the set of all fixed points of operator \eqref{A1}   is Fix$(W)=F_{11}\cup F_{12}$, where
$$F_{11}=\Big\{(0,a,u,v):
\quad u+v=b, \quad u, \, v \in[0,b] \Big\}$$ and
$$F_{12}=\Big\{(x,y,b,0):\quad x+y=a, \quad x,\,y\in[0,a]\Big\}.$$

\begin{defn}
A fixed point $s$ of the operator $W$ is called hyperbolic if its
Jacobian $J$ at $s$ has no eigenvalues on the unit circle.
\end{defn}

\begin{defn} A hyperbolic fixed point $s$ is called:

i) attracting if all the eigenvalues of the Jacobi matrix $J(s)$
are less than 1 in absolute value;

ii) repelling if all the eigenvalues of the Jacobi matrix $J(s)$
are greater than 1 in absolute value;

iii) a saddle otherwise.
\end{defn}

It is not hard to see that
$\lambda_1=0$, $\lambda_2=1$, $\lambda_3=1-\frac{v}{b}$ and
$\lambda_1=0$, $\lambda_2=1$, $\lambda_3=1-\frac{x}{a}$ are eigenvalues
of the fixed points of the forms $F_{11}$ and $F_{12}$ respectively.
By these definitions we see that all fixed points of the operator
\eqref{A1} are nonhyperbolic fixed points.

\section{Limit set}
Denote
$${\partial{S^{2,2}}}=\{t=(x,y,u,v)\in S^{2,2}: xyuv=0\}.$$

Take any initial point $t=(x,y,u,v)\in{\partial{S^{2,2}}}$. Consider the
following subsets of ${\partial{S^{2,2}}}$.
\begin{center}
$E_1=\big\{(x,y,u,v)\in{\partial{S^{2,2}}}: x=0 \big\}$, \\
$E_2=\big\{(x,y,u,v)\in{\partial{S^{2,2}}}: y=0 \big\}$,\\
$E_3=\big\{(x,y,u,v)\in{\partial{S^{2,2}}}: u=0 \big\}$,\\
$E_4=\big\{(x,y,u,v)\in{\partial{S^{2,2}}}: v=0 \big\}$.
\end{center}

If $t=(x,y,u,v)\in{E_1}$ then $W(t)\in{F_{11}}$. If $t=(x,y,u,v)\in{E_4}$
then $W(t)\in{F_{12}}$. \\When $t=(x,y,u,v)\in{E_2}$ then $W(t)\in{E_4}$ and
$W^{2}(t)\in{F_{12}}$. \\When $t=(x,y,u,v)\in{E_3}$ then $W(t)\in{E_1}$ and $W^{2}(t)\in{F_{11}}$.

Now we take any initial point $t=(x,y,u,v)\in{S^{2,2}}\backslash{\partial{S^{2,2}}}$.

Introduce the following notations

\begin{equation}\label{A9}
\alpha=\frac{x}{x+y}, \quad \beta=\frac{v}{u+v}, \quad
\alpha'=\frac{x'}{x'+y'}, \quad \beta'=\frac{v'}{u'+v'},
\end{equation}

which yields the nonlinear dynamical system
\begin{equation}\label{A10}
V:
\begin{cases}
\displaystyle
\alpha'=\frac{\alpha(1-\beta)}{1+(p_1-1)\alpha\beta}, \\[2ex]
\displaystyle
\beta'=\frac{\beta(1-\alpha)}{1+(p_2-1)\alpha\beta}
\end{cases}
\end{equation}
with the initial point $(\alpha^{(0)}, \beta^{(0)})\in\Delta$, where
\begin{equation}
\Delta:=\{(\alpha,\beta)\in{\mathbb{R}^{2}}: 0\leq\alpha\leq1,\,\,
0\leq\beta\leq1\}=[0,1]^2,
\end{equation}
and $$p_1=\frac{\sigma_1}{a}, \quad \quad p_2=\frac{\sigma_2}{b}. $$
There are three cases for $p_1$, $p_2$.
\begin{equation}\label{cases}
\begin{aligned}
& 1. \quad p_1=p_2=1, \\
& 2. \quad p_1>1>p_2>0, \\
& 3. \quad p_2>1>p_1>0.
\end{aligned}
\end{equation}

In order to find the fixed points of the operator \eqref{A10} we solve
the following system of equations for $(\alpha,\beta)$
\begin{equation}\label{A11}
\left\{
\begin{array}{ll}
\alpha(1+(p_1-1)\alpha\beta)=\alpha(1-\beta), \\[2ex]
\beta(1+(p_2-1)\alpha\beta)=\beta(1-\alpha).
\end{array}
\right.
\end{equation}
This system of equations gives us $\alpha\cdot\beta=0$, that is
$s_1=(\alpha,0)$ and $s_2=(0,\beta)$ are fixed points for the
operator \eqref{A10} where $\alpha\geq0$, $\beta\geq0$.

Using the system of equations \eqref{A10} we obtain

\begin{equation}\label{A12}
\begin{aligned}
\alpha^{(m+1)}=\frac{\alpha^{(m)}(1-\beta^{(m)})}{1+(p_1-1)
\alpha^{(m)}\beta^{(m)}}, \\[2ex]
\displaystyle
\beta^{(m+1)}=\frac{\beta^{(m)}(1-\alpha^{(m)})}{1+(p_2-1)
\alpha^{(m)}\beta^{(m)}}.
\end{aligned}
\end{equation}

\begin{lemma}\label{cs}
For any initial point $(\alpha,\beta)\in[0,1]^2$ it holds that
$$0\leq\alpha^{(m+1)}\leq\alpha^{(m)},  \quad 0\leq\beta^{(m+1)}\leq\beta^{(m)}.$$
In particular, the sequences
$\alpha^{(m)}=\frac{x^{(m)}}{x^{(m)}+y^{(m)}}, \ \ m\geq1$
and $\beta^{(m)}=\frac{v^{(m)}}{u^{(m)}+v^{(m)}}, \ \ m\geq1$
are convergent.
\end{lemma}
\begin{proof} Since $V:[0,1]^2\rightarrow[0,1]^2$ and for any $m\in \mathbb{N}$

\begin{center}
$1+(p_1-1)\alpha^{(m)}\in[\min\{1,p_1\}; \max\{1,p_1\}]$, \\
$1+(p_2-1)\beta^{(m)}\in[\min\{1,p_2\}; \max\{1,p_2\}]$, \\
$1+(p_1-1)\alpha^{(m)}\beta^{(m)}\in[\min\{1,p_1\}; \max\{1,p_1\}]$, \\
$1+(p_2-1)\alpha^{(m)}\beta^{(m)}\in[\min\{1,p_2\}; \max\{1,p_2\}]$
\end{center} then it holds that

$$\alpha^{(m+1)}-\alpha^{(m)}=\frac{-\alpha^{(m)}\beta^{(m)}(1+(p_1-1)\alpha^{(m)})}{1+(p_1-1)\alpha^{(m)}\beta^{(m)}}\leq0,$$
and that
$$\beta^{(m+1)}-\beta^{(m)}=\frac{-\alpha^{(m)}\beta^{(m)}(1+(p_2-1)\beta^{(m)})}{1+(p_2-1)\alpha^{(m)}\beta^{(m)}}\leq0.$$
This completes the proof.
\end{proof}
\begin{thm}\label{ta}
For any initial point $(x,y,u,v)\in{S^{2,2}}$
the sequence $$W^m(x,y,u,v)=(x^{(m)},y^{(m)},u^{(m)},v^{(m)})$$ is convergent and
$$\lim\limits_{m\rightarrow\infty}x^{(m)}\cdot{v^{(m)}}=0.$$
\end{thm}
\begin{proof} By Lemma \ref{cs} all trajectories of the operator \eqref{A10} have a limit point
and since the operator is continuous, each trajectory converges to a fixed point $s_1=(\alpha,0)$
or $s_2=(0,\beta)$. Therefore we have always
$$\alpha^{(m)}\cdot\beta^{(m)}\rightarrow{0}, \ \ \mbox{as} \ \ m\rightarrow\infty.$$

In a view of \eqref{A1} and \eqref{A9} we get
\begin{equation}\label{A13}
\left\{
\begin{array}{ll}
\begin{aligned}
& x^{(m+1)}  &= \, &\frac{ax^{(m)}u^{(m)} }{(x^{(m)}+y^{(m)})
(u^{(m)}+v^{(m)})} &=&a\alpha^{(m)}(1-\beta^{(m)}), \\[2mm]
& y^{(m+1)}  &= \, &\frac{\sigma_1x^{(m)}v^{(m)}+ay^{(m)}u^{(m)}+
ay^{(m)}v^{(m)}}{(x^{(m)}+y^{(m)})(u^{(m)}+v^{(m)})} &=& \sigma_1
\alpha^{(m)}\beta^{(m)}+a(1-\alpha^{(m)}), \\[2mm]
& u^{(m+1)}  &= \, &\frac{\sigma_2x^{(m)}v^{(m)}+bx^{(m)}u^{(m)}+
by^{(m)}u^{(m)}}{(x^{(m)}+y^{(m)})(u^{(m)}+v^{(m)})} &=& \sigma_2
\alpha^{(m)}\beta^{(m)}+b(1-\beta^{(m)}), \\[2mm]
& v^{(m+1)}  &= \, &\frac{by^{(m)}v^{(m)}}{(x^{(m)}+y^{(m)})
(u^{(m)}+v^{(m)})} &=& b\beta^{(m)}(1-\alpha^{(m)}).
\end{aligned}
\end{array}
 \right.
\end{equation}
This completes the proof.
\end{proof}
Define the following sets:
\begin{align*}
T_0 &= \Big\{(x,y,u,v)\in{S^{2,2}}: \, \lim\limits_
{m\rightarrow\infty}x^{(m)}=\lim\limits_
{m\rightarrow\infty}v^{(m)}=0\  \Big\}, \\
T_1 &= \Big\{(x,y,u,v)\in{S^{2,2}}: \, \lim\limits_
{m\rightarrow\infty}v^{(m)}=0, \quad \lim\limits_
{m\rightarrow\infty}x^{(m)}\in(0,a]  \Big\}, \\
T_2 &= \Big\{(x,y,u,v)\in{S^{2,2}}: \, \lim\limits_
{m\rightarrow\infty}x^{(m)}=0, \quad \lim\limits_
{m\rightarrow\infty}v^{(m)}\in(0,b] \Big\}.
\end{align*}
If $t=(x,y,u,v)\in{T_0}$, then
\begin{equation}
\lim\limits_{m\rightarrow\infty}\beta^{(m)}=\lim\limits_{m\rightarrow\infty}
\alpha^{(m)}=0.
\end{equation}
and \eqref{A13} shows that for any initial point
$t=(x,y,u,v)\in{T_0}$ for the trajectories of the operator \eqref{A1}
we have $$W_1^{m}=(x^{(m)},y^{(m)},u^{(m)},v^{(m)})\rightarrow
{(0,a,b,0)} \quad \mbox{as} \, \,  m \, \,\mbox{tends \, to} \, \, \infty.$$
If $t=(x,y,u,v)\in{T_1}$, then
\begin{equation}
\lim\limits_{m\rightarrow\infty}\beta^{(m)}=0 \ \
\mbox{and} \ \  \lim\limits_{m\rightarrow\infty}
\alpha^{(m)}=\alpha_0\in(0,1].
\end{equation}
System of equations \eqref{A13} shows that for any initial point
$t=(x,y,u,v)\in{T_1}$ for the trajectories of the operator \eqref{A1}
we have $$W_1^{m}=(x^{(m)},y^{(m)},u^{(m)},v^{(m)})\rightarrow
{\Big(a\alpha_0, a(1-\alpha_0),b,0\Big)}\in{F_{12}}
\quad \mbox{as} \, \,  m \, \,\mbox{tends \, to} \, \, \infty.$$
If $t=(x,y,u,v)\in{T_2}$, then
\begin{equation}
\lim\limits_{m\rightarrow\infty}\alpha^{(m)}=0 \ \
\mbox{and} \ \  \lim\limits_{m\rightarrow\infty}
\beta^{(m)}=\beta_0\in(0,1].
\end{equation}
System of equations \eqref{A13} shows that for any initial point
$t=(x,y,u,v)\in{T_2}$ for the trajectories of the operator \eqref{A1}
we have $$W_1^{m}=(x^{(m)},y^{(m)},u^{(m)},v^{(m)})\rightarrow
{\Big(0,a,b(1-\beta_0),b\beta_0\Big)}\in{F_{11}}
\quad \mbox{as} \, \,  m \, \,\mbox{tends \, to} \, \, \infty.$$
Therefore we have the following
\begin{cor} For any initial point $t=(x,y,u,v)\in{S^{2,2}}$
the $\omega$-limit set $\omega(t)$ of the operator \eqref{A1}  consists a single point and
\begin{equation}
\omega(t)\in\left\{\begin{array}{ll}
\{(0,a,b,0)\} \ \ \mbox{if} \ \ t=(x,y,u,v)\in{T_0},\\[2mm]
{F_{12}} \ \ \mbox{if} \ \ t=(x,y,u,v)\in{T_1},\\[2mm]
{F_{11}} \ \ \mbox{if} \ \ t=(x,y,u,v)\in{T_2}.
\end{array}\right.
\end{equation}
\end{cor}
\begin{defn}
An operator $W$ is called regular if for any initial point $s^{(0)} \in S^{2,2}$, the limit
\[\lim_{m\to \infty}W^m(s^{(0)}) \]
exists.
\end{defn}
The following is a corollary of Theorem \ref{ta}.
\begin{cor} The operator \eqref{A1}  is regular.
\end{cor}

We would like to describe the sets $T_0$, $T_1$ and $T_2$ implicitly.
\subsection{Case 1}
Let we have $$p_1=p_2=1.$$ Then operator \eqref{A10} looks like:
\begin{equation}\label{V1}
V_1:\left\{\begin{array}{ll}
\alpha'=\alpha-\alpha\beta\\[2mm]
\beta'=\beta-\alpha\beta
\end{array}\right.
\end{equation}
where $(\alpha;\beta)\in\Delta$.

$s_1=(\alpha,0)$ and $s_2=(0,\beta)$ are non-hyperbolic
fixed points of \eqref{V1} with the eigenvalues $\lambda_1=1$,
$\lambda_2=1-\alpha\in[0,1]$ and $\lambda_1=1$,
$\lambda_2=1-\beta\in[0,1]$ respectively.

We say the set $E$ is invariant respect to the operator $V$ if $V(E)\subset{E}$.
\begin{lemma}The following sets
$$M_0=\big\{(\alpha,\beta)\in[0,1]^2: \quad
\beta=\alpha\big\}$$
$$M_1=\big\{(\alpha,\beta)\in[0,1]^2: \quad
\beta<\alpha\big\}$$ and $$M_2=\big\{(\alpha,\beta)
\in[0,1]^2: \quad \beta>\alpha\big\}$$
are invariant sets respect to the operator \eqref{V1}.
\end{lemma}
\begin{proof} Straightforward.
\end{proof}
We look for the invariant curves of the operator
\eqref{V1}. Let $\beta=g(\alpha)$ be an invariant
curve then $\beta'=g(\alpha')$ and to find invariant
curve leads to solve the following iterative functional
equation
\begin{equation}\label{iterative f.e.1}
f(\alpha)\big(\alpha-f(\alpha)\big)(1-\alpha)
=\alpha\big(f(\alpha)-f(f(\alpha))\big)
\end{equation}
where
$f(\alpha)=\alpha\big(1-g(\alpha)\big)$ which is not identically zero.

We solve \eqref{iterative f.e.1} in the space $C^\infty{[0,1]}$.

The equation \eqref{iterative f.e.1} gives $f(0)=0$.
Moreover from $f\in{C^\infty{[0,1]}}$ we get
\begin{equation}\label{f}
f(\alpha)=\sum_{k=1}^{\infty}c_k\alpha^k
\end{equation}
and
\begin{equation}\label{ff}
f(f(\alpha))=\sum_{k=1}^{\infty}c_kf^k(\alpha)=
\sum_{k=1}^{\infty}d_k\alpha^k
\end{equation}
where $$d_k=\sum_{l=1}^
{k}c_l\big(\sum_{i_1+i_2+...+i_l=k}c_{i_1}\cdot{c_{i_2}}\cdot...
\cdot{c_{i_l}}\big).$$
\begin{thm}
The solutions of the functional equation \eqref{iterative f.e.1}
are $$ f(\alpha)=\alpha \ \ \mbox{and} \ \ f(\alpha)=\theta\alpha-\alpha^2
$$ where $\theta$ is an arbitrary constant.

In particular $$g(\alpha)=0 \ \ \mbox{and} \ \ g(\alpha)=
\alpha+1-\theta$$ are the only invariant curves of the operator
\eqref{V1}.
\end{thm}
\begin{proof} Substituting \eqref{f} and \eqref{ff} to the
\eqref{iterative f.e.1} we obtain
\begin{equation}\label{identity}
\sum_{k=1}^{\infty}a_k\alpha^k\equiv\sum_{k=1}^{\infty}b_k\alpha^k,
\end{equation}
which is equivalently to
\begin{equation}\label{a_k=b_k}
a_k=b_k \ \ \mbox{for all} \ \ k=1,2,...
\end{equation}
where
$$b_k=\sum_{l=1}^{k}c_{l+1}\big(\sum_{i_1+i_2+...+i_l=k}c_{i_1}
\cdot{c_{i_2}}\cdot...\cdot{c_{i_l}}\big)$$ and
\begin{equation}\label{a_k}
a_k=\left\{
\begin{array}{ll}
-1-c_2+c_1 \ \ \mbox{if} \ \ k=1,\vspace{1,5mm} \cr
c_{k+1}-c_k \ \ \mbox{if} \ \ k=2,3,...\vspace{1,5mm} \cr
\end{array}
 \right.
\end{equation}
From identity of \eqref{identity} for $k=1$ it holds
that $$(1-c_1)(1+c_2)=0.$$ For $k=2$ we see that
$$c_3(1-c_1^2)=c_2(1+c_2).$$ For $k=3$ we see that
$$c_4(1-c_1^3)=c_3(1+c_2+2c_1c_2).$$ These last three
equations imply that $$c_1=1, \, c_2=0, \, c_3=0 \,\, \mbox{or} \,\,
c_1 \,\,  \mbox{is arbitrary}, \, c_2=-1, \,  c_3=0.$$
Now we show by induction that $c_k=0$ for all $k=3,4,...$.

Suppose  $c_k=0$ for all $k=3,4,...n$. Then putting
$k=n$, $k=n+1$ and $k=n+2$ in \eqref{a_k=b_k} we get
$$c_{n+1}-c_{n}=\sum_{l=1}^{n}c_{l+1}\big(\sum_
{i_1+i_2+...+i_l=n}c_{i_1}\cdot{c_{i_2}}\cdot...\cdot
{c_{i_l}}\big)$$
$$c_{n+2}-c_{n+1}=
\sum_{l=1}^{n+1}c_{l+1}\big(\sum_{i_1+i_2+...+i_l=n+1}
c_{i_1}\cdot{c_{i_2}}\cdot...\cdot{c_{i_l}}\big)$$ and
$$c_{n+3}-c_{n+2}=
\sum_{l=1}^{n+2}c_{l+1}\big(\sum_{i_1+i_2+...+i_l=n+2}
c_{i_1}\cdot{c_{i_2}}\cdot...\cdot{c_{i_l}}\big).$$
 These
last three equations equivalent to
\begin{equation}\label{c_1}
c_{n+1}=c_{n+1}c_1^n
\end{equation}
\begin{equation}\label{c_2}
c_{n+2}(1-c_1^{n+1})=c_{n+1}(1+c_2+nc_2c_1^{n-1})
\end{equation}
and
\begin{equation}\label{c_3}
c_{n+3}(1-c_1^{n+2})=c_{n+2}+c_2c_{n+2}+\frac{n!}{2!}c_{n+1}c_2^2c_1^{n-2}+(n+1)c_{n+2}c_2c_1^{n}
\end{equation}
If $c_1\neq\pm1$ or $c_1=-1$ and $n$ is odd then \eqref{c_1} gives $c_{n+1}=0$. If $c_1=1$ then
\eqref{c_2} gives $c_{n+1}=0$, otherwise if $c_1=-1$ and $n$ is even then from
\eqref{c_2} and \eqref{c_3} we come to
\begin{equation*}
c_{n+1}[(1+(1-n)c_2)(1+(n+2)c_2)+n!c_2^2]=0
\end{equation*}
which shows again that $c_{n+1}=0$.
Thus for all $k=3,4,...$ we have
$c_k=0$. That is $$f(\alpha)=\alpha \ \ \mbox{and} \ \
f(\alpha)=\theta\alpha-\alpha^2$$ are solutions of the
iterative functional equation \eqref{iterative f.e.1},
where $\theta$ is an arbitrary constant.

This completes the proof.
\end{proof}
So, we have proved that
$$\gamma_\theta=\big\{(\alpha,\beta)\in[0,1]^2:
\beta=g(\alpha)=\alpha+1-\theta\}, \ \ \theta\in[0,2]$$ is one-parametric family of
invariant curves.

Note that  $$ \bigcup_{\theta\in[0,1)}\gamma_\theta=M_2, \quad
\bigcup_{\theta\in(1,2]}\gamma_\theta=M_1, \quad \gamma_1=M_0$$ and
$$\gamma_{\theta_1}\cap\gamma_{\theta_2}=\emptyset \ \ \mbox{for} \ \ \mbox{any} \ \ \theta_1\neq\theta_2.$$
Thus it suffices to study the dynamical system on each invariant curve $\gamma_\theta$. We have the following result (See Figure \ref{p1}).
\begin{thm}\label{als} The following assertions hold
\begin{itemize}
\item[(i)] If $\theta=1$ then for any initial point $t=(\alpha,\beta)\in{M_0}$, (i.e. $\alpha=\beta$) we have
$$\lim\limits_{m\rightarrow\infty}V_1^{(m)}(\alpha,\beta)=
\lim\limits_{m\rightarrow\infty}(\alpha^{(m)},\beta^{(m)})=(0;0).$$
\item[(ii)] If $\theta\in(1,2]$ then for any initial point $t=(\alpha,\beta)\in{\gamma_\theta}$ we have
$$\lim\limits_{m\rightarrow\infty}V_1^{(m)}(\alpha,\beta)=
\lim\limits_{m\rightarrow\infty}(\alpha^{(m)},\beta^{(m)})=(\theta-1;0).$$
\item[(iii)] If $\theta\in[0,1)$ then for any initial point $t=(\alpha,\beta)\in{\gamma_\theta}$ we have
$$\lim\limits_{m\rightarrow\infty}V_1^{(m)}(\alpha,\beta)=
\lim\limits_{m\rightarrow\infty}(\alpha^{(m)},\beta^{(m)})=(0;1-\theta).$$
\end{itemize}
\end{thm}

\begin{figure}[h]
  \centering
  \includegraphics[width=6 cm]{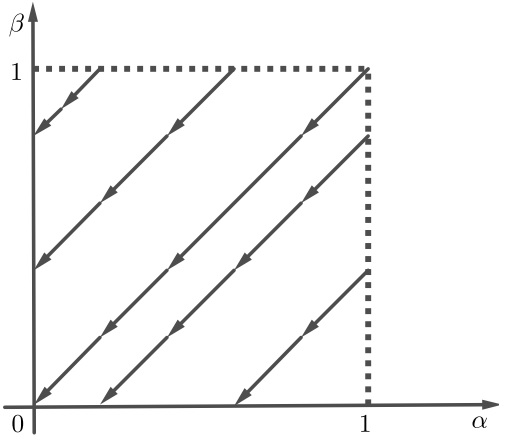}\\
  \caption{Dynamics of the operator \eqref{V1} on the invariant lines $\gamma_\theta$.
  The trajectory converges to the fixed point on the intersection of the line and the axes $O\alpha$ or $O\beta$.}\label{p1}
\end{figure}

Going back to the old variables $(x,y,u,v)$, when $p_1=p_2=1$ we obtain $\sigma_1=a$, $\sigma_2=b$ and
$$\Omega_\theta=\big\{(x,y,u,v)\in{S^{2,2}}: \quad
\frac{v}{u+v}=\frac{x}{x+y}+1-\theta\big\}$$ is an invariant surface respect
to the operator \eqref{A1} and it holds that
 $$ \bigcup_{\theta\in[0,1)}\Omega_\theta=T_2=
\Big\{(x,y,u,v)\in{S^{2,2}}: yv>xu \Big\},$$
$$\bigcup_{\theta\in(1,2]}\Omega_\theta=T_1=
\Big\{(x,y,u,v)\in{S^{2,2}}: yv<xu \Big\},$$
$$\Omega_1=T_0=\Big\{(x,y,u,v)\in{S^{2,2}}: yv=xu \Big\}$$
and
$$\Omega_{\theta_1}\cap\Omega_{\theta_2}=\emptyset \ \ \mbox{for any} \ \ \theta_1\neq\theta_2.$$
Thus it suffices to study the dynamical system on each invariant surfaces $\Omega_\theta$. As a corollary of Theorem \ref{als} we have the following

\begin{thm}\label{au} The following assertions hold
\begin{itemize}
\item[(i)] For any initial point $t=(x,y,u,v)\in{T_0}$, we have
$$\lim\limits_{m\rightarrow\infty}W^{(m)}(x,y,u,v)=
\lim\limits_{m\rightarrow\infty}(x^{(m)},y^{(m)},u^{(m)},v^{(m)})=(0;a;b;0).$$
\item[(ii)] If $\theta\in(1,2]$ then for any initial point $t=(x,y,u,v)\in \Omega_\theta$ the following holds
$$\lim\limits_{m\rightarrow\infty}W^{(m)}(x,y,u,v)=
\lim\limits_{m\rightarrow\infty}(x^{(m)},y^{(m)},u^{(m)},v^{(m)})=(a(\theta-1);a(2-\theta);b;0).$$
\item[(iii)] If $\theta\in[0,1)$ then for any initial point $t=(x,y,u,v)\in \Omega_\theta$ the following holds
$$\lim\limits_{m\rightarrow\infty}W^{(m)}(x,y,u,v)=
\lim\limits_{m\rightarrow\infty}(x^{(m)},y^{(m)},u^{(m)},v^{(m)})=(0;a;b\theta;b(1-\theta)).$$
\end{itemize}
\end{thm}
\begin{cor} The operator \eqref{A1} has infinitely many
fixed points and for each such fixed point there is nonintersecting
trajectories which converge to the fixed points.
\end{cor}
\subsection{Case 2}
Let we have $$p_1>1>p_2>0,$$ or $$p_2>1>p_1>0.$$
\begin{lemma}\label{A15}
The set $$M_1=\big\{(\alpha,\beta)\in\Delta: \quad \beta\geq\alpha\big\}$$ is an invariant
set respect to the operator \eqref{A10} when $p_1>1>p_2>0$.

The set $$M_2=\big\{(\alpha,\beta)\in\Delta: \quad \beta\leq\alpha\big\}$$ is an invariant
set respect to the operator \eqref{A10} when $p_2>1>p_1>0$.
\end{lemma}
\begin{proof} Straightforward.
\end{proof}
In this case to find invariant
curves for the operator \eqref{A10} leads to solve the following iterative functional
equation
\begin{equation}\label{iterative f.e.2}
\begin{aligned}
&f(\alpha)\big(\alpha-f(\alpha)\big)(1-\alpha)[1+(p_1-1)f(f(\alpha))]
=\\ &\alpha\big(f(\alpha)-f(f(\alpha))\big)[1+(p_2-1)\alpha+(p_1-p_2)f(\alpha)]
\end{aligned}
\end{equation}
where
$f(\alpha)={\alpha\big(1-g(\alpha)\big)}\setminus{[1+(p_1-1)\alpha{g(\alpha)}]}$ which is not identically zero.

As above when we search the solution of the last functional equation in the space
$C^\infty{[0,1]}$ we get
\begin{equation}\label{identity 2}
\sum_{k=1}^{\infty}a_k\alpha^k\equiv\sum_{k=1}^{\infty}b_k\alpha^k,
\end{equation}
which is equivalently to
\begin{equation}\label{a_k=b_k 2}
a_k=b_k \ \ \mbox{for all} \ \ k=1,2,...
\end{equation}
where
$$a_k=\sum_{j=0}^{k}e_jq_{k-j}, \quad b_k=\sum_{j=0}^{k}n_jm_{k-j}$$
and

\begin{equation}\label{e_j}
e_j=\left\{
\begin{array}{ll}
1-c_1 \ \ \mbox{if} \ \ j=0, \vspace{1,5mm} \cr
-\sum_{l=1}^{j}c_{l+1}\big(\sum_{i_1+i_2+...+i_l=j}c_{i_1}
\cdot{c_{i_2}}\cdot...\cdot{c_{i_l}}\big) \ \ \mbox{if} \ \ j=1,2,...,k. \vspace{1,5mm} \cr
\end{array}
 \right.
\end{equation}

\begin{equation}\label{q_j}
q_j=\left\{
\begin{array}{ll}
1 \ \ \mbox{if} \ \ j=0, \vspace{1,5mm} \cr
(p_2-1)+(p_1-p_2)c_1 \ \ \mbox{if} \ \ j=1, \vspace{1,5mm} \cr
(p_1-p_2)c_j \ \ \mbox{if} \ \ j=2,3,...,k. \vspace{1,5mm} \cr
\end{array}
 \right.
\end{equation}

\begin{equation}\label{n_j}
n_j=\left\{
\begin{array}{ll}
1 \ \ \mbox{if} \ \ j=0, \vspace{1,5mm} \cr
(p_1-1)\sum_{l=1}^{j}c_{l}\big(\sum_{i_1+i_2+...+i_l=j}c_{i_1}
\cdot{c_{i_2}}\cdot...\cdot{c_{i_l}}\big) \ \ \mbox{if} \ \ j=1,2,...,k. \vspace{1,5mm} \cr
\end{array}
 \right.
\end{equation}

\begin{equation}\label{m_j}
m_j=\left\{
\begin{array}{ll}
1-c_1 \ \ \mbox{if} \ \ j=0, \vspace{1,5mm} \cr
-1+c_1-c_2 \ \ \mbox{if} \ \ j=1, \vspace{1,5mm} \cr
c_j-c_{j+1} \ \ \mbox{if} \ \ j=2,3,...,k. \vspace{1,5mm} \cr
\end{array}
 \right.
\end{equation}
Substituting \eqref{e_j}, \eqref{q_j}, \eqref{n_j}, \eqref{m_j} to the
\eqref{a_k=b_k 2} then when $k\geq3$ we attain recurrence formula for $c_k$

\begin{equation}
\begin{aligned}
&c_{k}(1-c_{1}^{k-1})=\sum_{j=1}^{k-2}\big[(c_{k-j-1}-c_{k-j})(p_1-1)d_j+(p_1-p_2)c_{k-j}d'_j\big]
\\
&+c_{k-1}-(p_1-1)d_{k-2}+(1-c_1)(p_1-1)d_{k-1}(1-c_1)(p_1-p_2)c_k\\
&+(p_2-1)d'_{k-2}-\sum_{l=1}^{k-1}c_{l+1}\big(\sum_{i_1+i_2+...+i_l=k}c_{i_1}\cdot{c_{i_2}}\cdot...
\cdot{c_{i_l}}\big)
\end{aligned}
\end{equation}
where $c_k$ is the coefficient at \eqref{f} and $$d_j=\sum_{l=1}^
{j}c_l\big(\sum_{i_1+i_2+...+i_l=j}c_{i_1}\cdot{c_{i_2}}\cdot...
\cdot{c_{i_l}}\big), \quad d'_j=\sum_{l=1}^
{j}c_{l+1}\big(\sum_{i_1+i_2+...+i_l=j}c_{i_1}\cdot{c_{i_2}}\cdot...
\cdot{c_{i_l}}\big).$$

We were not able to solve these systems for the coefficients. Therefore the following is an open problem:

{\bf Open problem.} Describe all solutions of the functional equation (\ref{iterative f.e.2}).\\

 Numerical analysis  shows that (see at the Figure \ref{p2} and Figure \ref{p3}) in the cases $p_1>1>p_2>0$ (resp. in the case $p_2>1>p_1>0$)
 there are nonintersecting concave (resp. convex) invariant curves and the trajectory started on an invariant curve converges to the intersecting point of
 the invariant curve and the axes $O\alpha$ or $O\beta$.
\begin{figure}
  \centering
  \includegraphics[width=13 cm]{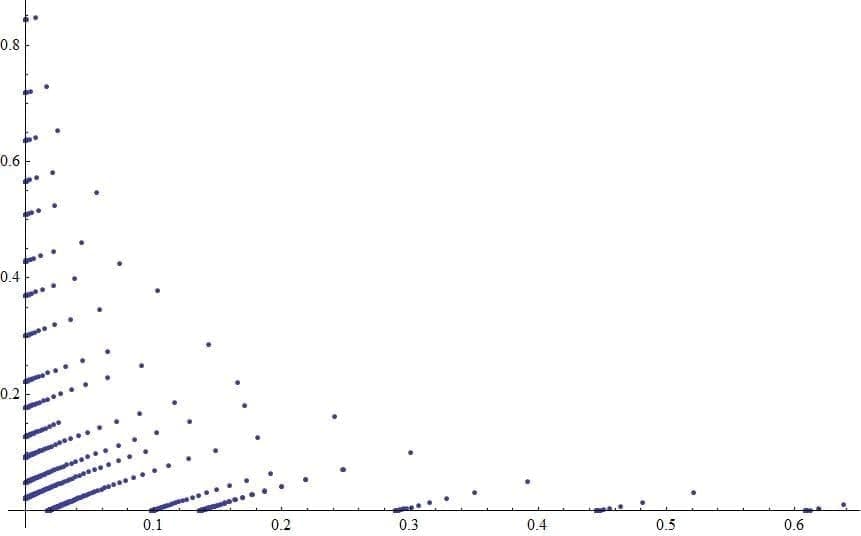}\\
  \caption{Dynamics of the operator \eqref{A10} on invariant concave curves for the case $p_1>1>p_2>0$.
  The trajectory converges to the fixed point on the intersection of the invariant curve and the axes $O\alpha$ or $O\beta$.}\label{p2}
\end{figure}
\begin{figure}
  \centering
  \includegraphics[width=13 cm]{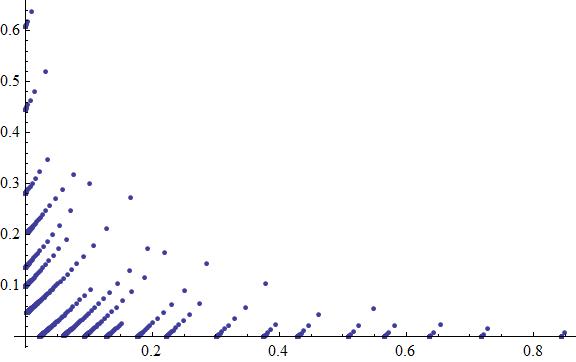}\\
  \caption{Dynamics of the operator \eqref{A10} on invariant convex curves for the case $p_2>1>p_1>0$.
  The trajectory converges to the fixed point on the intersection of the invariant curve and the axes $O\alpha$ or $O\beta$.}\label{p3}
\end{figure}
These numerical analysis and the above considered particular cases allowed us to make the following

{\bf Conjecture.} If $p_1>1>p_2>0$ (or $p_2>1>p_1>0$) then for each fixed point $p\in {\rm Fix}(W)$ there exists unique invariant surface $\Gamma_p\subset S^{2,2}$, such that for any initial point $s^{(0)}\in \Gamma_p$  the limit of its trajectory (under operator (\ref{A1})) converges to the fixed point $p$.
Moreover,
$$\bigcup_{p\in {\rm Fix}(W)}\Gamma_p=S^{2,2}.$$

\section{Conclusion}
 Let $s^{(0)} = (x, y, u, v)\in S^{2,2}$ be an
initial state, i.e. the probability distribution on the set of female and male types.

The following are interpretations of our results:

\begin{itemize}
\item The set of all fixed points is subset of the boundary of $S^{2,2}$ means that at least one
type of female or male  in future of population will surely disappear.

\item The existence of invariant curves (in particular lines) means that if states of the population initially satisfied
a relation (described the invariant set) then the future of the population remains in the same relation.

\item Regularity of the operator means that for any initial state of the population we can explicitly determine its limit (final) state.

\item For any $s^{(0)}\in T_0$  as time goes to infinity the type 1 of female and type 2 of males will disappear (die).

\item For any $s^{(0)}\in T_1$  as time goes to infinity the type 2 of males will disappear.

\item For any $s^{(0)}\in T_2$  as time goes to infinity the type 1 of females will disappear.

   \end{itemize}

\end{document}